\documentclass[11pt]{article}       % onecolumn (second format)
\usepackage{graphicx}
\usepackage{mathptmx}      % use Times fonts if available on your TeX system
\usepackage{mathtools}
\usepackage{times}
\usepackage{mathtools}
\usepackage{amsmath,amssymb,amsthm}

\usepackage{bbm}
\usepackage{algorithmic}
\usepackage{algorithm}
\usepackage{color}
\usepackage[title]{appendix}
\newcommand{\Rb}{\mathbbm{R}}      % for Real numbers

\newcommand{\Dc}{\mathcal{D}}

\newcommand{\Eb}{\mathbbm{E}}
\newcommand{\Fc}{\mathcal{F}}

\newcommand{\Kc}{\mathcal{K}}
\newcommand{\Lc}{\mathcal{L}}

\newcommand{\Sc}{\mathcal{S}}

\newcommand{\Wc}{\mathcal{W}}
\newcommand{\Zc}{\mathcal{Z}}

\newcommand{\argmin}{\mathop{\rm argmin}}

\newcommand{\dist}{\mathop{\rm dist}}

\newcommand*{\dt}[1]{\overset{\hbox{\,\tiny${}_\bullet$}}{#1}}

\newtheorem{proposition}{Proposition}[section]
\newtheorem{theorem}[proposition]{Theorem}
\newtheorem{corollary}[proposition]{Corollary}
\newtheorem{lemma}[proposition]{Lemma}
\newtheorem{definition}[proposition]{Definition}
\newtheorem{remark}[proposition]{Remark}
\newtheorem{example}[proposition]{Example}
\def\nicebox{\vbox{\hrule height0.6pt\hbox{%
   \vrule height1.3ex width0.6pt\hskip0.8ex
   \vrule width0.6pt}\hrule height0.6pt  }}
\newcommand{\conv}{\text{conv}}
\newcommand{\proj}{{\rm Proj}}

\newenvironment{tightlist}[1]{%
    \list{{\textup{(\roman{enumi})}}}{\settowidth\labelwidth{{\textup{(#1)}}}
    \leftmargin 0pt \advance\leftmargin\labelsep \itemindent \parindent
    \parsep 0pt plus 1pt minus 1pt \topsep 0pt \itemsep 0pt
    \usecounter{enumi}}}{\endlist}

\begin{document}

\title{A Stochastic Subgradient Method for Nonsmooth Nonconvex Multi-Level Composition Optimization
}


%\titlerunning{Stochastic Subgradient Method with Averaging}        % if too long for running head

\author{Andrzej Ruszczy\'nski\footnote{Rutgers University, Department of Management Science and Information Systems, Piscataway, NJ 08854, USA;
              email: {rusz{@}rutgers.edu}}
}


\date{January 29, 2020}

\maketitle

\maketitle

\begin{abstract}
We propose a single time-scale stochastic subgradient method for constrained
optimization of a composition of several nonsmooth and nonconvex functions.  The functions are  assumed to be locally Lipschitz and  differentiable in a generalized sense. Only stochastic estimates of the values and
generalized derivatives of the functions are used. The method is parameter-free. We prove convergence with probability one of the method, by associating with it a system of differential inclusions
and devising a nondifferentiable Lyapunov function for this system. For problems with functions having Lipschitz continuous derivatives,
 the method finds a point satisfying an optimality measure with error of order $N^{-1/2}$, after executing $N$ iterations with constant stepsize.

\emph{AMS}:
  90C15, 49J52, 62L20
\end{abstract}

\section{Introduction}

We consider the composition optimization problem
\begin{equation}
\label{main_prob}
\min_{x\in X}\; f_1\Big(x,f_2\big(x, \cdots f_{M-1}\big(x,f_M(x)\big)\cdots \big)\Big),
\end{equation}
where $X\subset \Rb^n$ is convex and closed, and $f_m:\Rb^n\times \Rb^{d_{m+1}}\to \Rb^{d_m}$, $m=1,\dots,M-1$, and
$f_M:\Rb^n\to \Rb^{d_M}$ are locally Lipschitz continuous functions, possibly neither convex nor smooth.

 The Clarke derivatives of $f_m(\cdot,\cdot)$ are not available;
 instead, we postulate access to their random estimates.
Such situations occur in \emph{stochastic composition optimization}, where
\begin{equation}
\label{stoch_prob}
f_m(x,u_{m+1}) = \Eb\big[ \varphi_m(x,u_{m+1},\xi_m)\big],\quad m =1,\dots,M,
\end{equation}
in which $\xi_m$ is a random vector, and $\mathbb{E}$ denotes the expected value. Two examples
illustrate the relevance of the problem.

\begin{example}[\it{Risk-Averse Optimization}]
\label{e:1}
{\rm Given some random loss function $H:\Rb^n\times \Omega\to \Rb^n$, we consider the random variable $Z = H(x)$ as an element of a Banach space $\Zc$ (for example,
$\Lc_2(\varOmega,\Fc,P)$) and evaluate its quality by the functional $\rho: \Zc \to \Rb$, called a \emph{risk measure}. This leads to the problem
\begin{equation}
\label{risk-opt}
\min_{x\in X} F(x) = \rho[H(x)].
\end{equation}
Particularly simple and useful is the \emph{mean--semideviation measure} \cite{OR:1999,OR:2001}, which
 is an example of a coherent measure of risk \cite{ADEH:1999} (see also \cite{follmer2011stochastic,shapiro2009lectures} and the references therein). It has the following form:
\begin{equation}
\label{msd}
\rho[Z] =  \Eb[Z]  + \varkappa\, \bigg(\Eb\Big[\Big(\max\big(0,Z - \Eb[Z]\big)\Big)^p\Big]\bigg)^{1/p},\quad p \ge 1.
\end{equation}
For $p=1$, problem \eqref{risk-opt}  with the risk measure \eqref{msd} can be written in the form \eqref{main_prob}
with two functions
\begin{align*}
f_1(x,u_2) &= \Eb\Big[H(x) +  \varkappa\, \max\big(0,H(x) - u_2\big)\Big],\\
f_2(x) &= \Eb[H(x)].
\end{align*}
For $p=2$, problem \eqref{risk-opt} with the risk measure \eqref{msd}  can be cast (with a minor regularization) in the form \eqref{main_prob}
with three functions
\begin{align*}
f_1(x,u_2) &= \Eb\big[H(x)\big] + \varkappa\,(\varepsilon+ u_2)^{1/2}, \quad \varepsilon>0,\\
f_2(x, u_3) &= \Eb\Big[\Big(\max\big(0,H(x) - u_3\big)\Big)^2\Big],\\
f_3(x) &= \Eb[H(x)].
\end{align*}
We added $\varepsilon>0$ into the definition of $f_1(\cdot,\cdot)$ to ensure local Lipschitz property for all $u_2\ge 0$.
Frequently, $H(\cdot)$ is non-convex and non-differentiable
in modern machine learning models.
}\hfill \nicebox
\end{example}

\begin{example}[\it{Stochastic Variational Inequality}]
\label{e:SVI}
{\rm
%The \emph{stochastic variational inequality} problem is formulated as follows.
We have a random mapping $H:\Rb^n\times \Omega\to \Rb^n$
on some probability space $(\varOmega,\Fc,P)$ and a closed convex set $X$.  The problem is to find ${x}\in X$ such that
\begin{equation}
\label{SVI}
\big\langle \mathbb{E}[H(x)], \xi - {x} \big\rangle  \le 0, \quad \text{for all}\quad \xi \in X.
\end{equation}
The reader is referred to the recent publications \cite{iusem2017extragradient} and \cite{koshal2013regularized}
for a discussion of the challenges associated with this problem and its applications to stochastic equilibria.
% { (our use of the ``$\le$'' relation instead of the common ``$\ge$''
%is only motivated by the easiness to show the conversion to our formulation)}.
We may reformulate problem \eqref{SVI} as \eqref{main_prob} by defining $f_1: \Rb^n \times \Rb^n\to \Rb$ as
\begin{equation}
f_1(x,u) = \max_{y \in X}\, \left\{ \langle u, y - x \rangle - \frac{r}{2} \| y - x\|^2 \right\},\quad r>0,\label{gad_function}
\end{equation}
and $f_2(x) = \mathbb{E}[H(x)]$.
In this case, we  have access to the gradient of $f_1$, but the value and the Jacobian of $f_2$ must be estimated. We do not require
$H(\cdot)$ to be monotone or differentiable.
}\hfill \nicebox
\end{example}
%In such situations,
%under fairly general asumtions one can obtain samples
%$(\tilde{\xi},\tilde{\zeta})$ of $(\xi,\zeta)$, and treat $\psi(x,\tilde{\zeta})$, and selected elements of
%$\partial_x\psi(x,\tilde{\zeta})$, and $\partial_u \varphi(u,\xi)$
%as random estimates of  $\mathbb{E}[\psi(x;\zeta)]$, and some elements of $\partial \mathbb{E}[\psi(x;\zeta)]$, and $\partial\mathbb{E}\big[ \varphi(u,\xi)\big]$, respectively. The symbol $\partial\!f(\cdot)$ will always denote the Clarke subdifferential (or Jacobian) of the
%function $f(\cdot)$.

The research on stochastic subgradient methods for nonsmooth and nonconvex functions started in the late 1970's: see
Nurminski \cite{Nurminski1979numerical} for weakly convex functions and a general methodology for studying convergence of non-monotonic methods,
Gupal \cite{gupal1979stochastic} for convolution smoothing (mollification) of Lipschitz functions
and resulting finite-difference methods, and  Norkin \cite{norkin:phd} and \cite[Ch. 3 and 7]{mikhalevich1987nonconvex}
for unconstrained problems with ``generalized differentiable'' functions.

Recently, by an approach via differential inclusions, Duchi and Ruan \cite{duchi2018stochastic} studied proximal methods for
 sum-composite problems with weakly convex functions,
 Davis \emph{et al.} \cite{davis2018stochastic} studied the subgradient method for locally Lipschitz Whitney $\mathcal{C}^1$-stratifiable functions
 with constraints, and Majewski \emph{et al.} \cite{majewski2018analysis} studied several methods for subdifferentially regular Lipschitz functions.

 The research on composition optimization problems started from
penalty functions for stochastic constraints and composite regression models in \cite{Ermoliev71} and \cite[Ch. V.4]{ermoliev1976methods}.
An established approach was to use two-level stochastic
recursive algorithms with two stepsize sequences in different time scales:  a slower one for updating the
main decision variable $x$, and a faster one for tracking the value of the inner function(s).
References \cite{wang2017stochastic,WaLiFa17} provide a detailed account of these techniques and existing results.
In \cite{yang2019multi} these ideas were extended to multilevel problems of form \eqref{main_prob}, albeit with
multiple time scales and under continuous differentiability assumptions.

A Central Limit Theorem for problem \eqref{stoch_prob} has been established in  \cite{dentcheva2017statistical}.  Large deviation bounds for the empirical optimal value were derived in \cite{ermoliev2013sample}.

The first single time-scale method for a two-level version ($M=2$) of problem \eqref{main_prob} with continuously differentiable functions has been recently proposed in \cite{ghadimi2018single}. It has  the complexity of ${\cal O}(1/\epsilon^2)$ to obtain an $\varepsilon$-solution
 of the problem, the same as methods for one-level unconstrained stochastic optimization. However, the construction of the method
 and its analysis depend on the Lipschitz continuity of the gradients of the functions involved, and its parameters
 depend on the corresponding Lipschitz constants.

To the best of our knowledge, there has been no research on stochastic subgradient methods for composition problems of form \eqref{main_prob}
where the functions involved may be neither convex nor smooth.

Our main objective is propose a single time-scale  method for solving the multiple composition problem \eqref{main_prob}, and to establish its convergence with probability one on
a broad class of problems, in which the functions $f_m(\cdot,\cdot)$, $m=1,\dots,M-1$, are assumed to be locally Lipschitz and admit
a chain rule, while $f_M(\cdot)$ may be only differentiable in a generalized sense  (to be defined in section \ref{s:2}). The class of such functions is broader than
the class of locally Lipschitz semismooth functions \cite{mifflin1977semismooth}. The method's  few parameters may be set to arbitrary positive constants. The main idea is to lift the problem to a higher dimensional space and to devise a single time scale
scheme for not only estimating the
solution, but also the values of all functions nested in \eqref{main_prob}, and the generalized gradient featuring in the optimality condition.

Our approach uses the differential equation method (see \cite{ljung1977analysis,kushner2012stochastic,kushner2003stochastic} and the references therein). Extension to differential inclusions
was proposed in  \cite{benaim2005stochastic,benaim2006stochastic} and further developed, among others,
in \cite{borkar2009stochastic,duchi2018stochastic,davis2018stochastic,majewski2018analysis}. In our analysis, we use a specially
tailored nonsmooth Lyapunov function for the method, which generalizes the idea of \cite{ruszczynski1986method,ruszczynski87}
to the composite setting.

The second contribution is the error analysis after finitely many iterations of the method. We prove that a non-optimality measure for problem \eqref{main_prob}, which corresponds
to the squared  norm of a gradient in the unconstrained one-level case, decreases at the rate $1/\sqrt{N}$, where $N$ is the number of iterations of the method. This matches
the best rate estimates for general unconstrained one-level problems \cite{GhaLan12}, and the statistical estimate of  \cite{dentcheva2017statistical} for plug-in estimates of composition risk functionals; it outperforms the estimate of the method of \cite{yang2019multi}.

The paper is organized as follows. In \S \ref{s:2}, we recall the main facts on generalized differentiation and the chain rule on a path. In \S \ref{s:3}, we describe our stochastic subgradient
method for problem \eqref{main_prob}. In \S \ref{s:4} we introduce relevant multifunctions and prove
the boundedness of the sequences generated by the method.
\S \ref{s:analysis} contains the proof of its  convergence for nonconvex and nonsmooth functions. Finally, in \S \ref{s:5}, we provide solution quality guarantees after finitely many iterations with
a constant stepsize, in the case when the functions have
Lipschitz continuous derivatives.

\section{Generalized Subdifferentials of Composite Functions}
\label{s:2}

We consider problem \eqref{main_prob} for locally Lipschitz continuous functions $f_m(\cdot)$
satisfying additional conditions of generalized differentiability, subdifferential regularity, or Whitney $\mathcal{C}^1$-stratification. Recall that
$f:\Rb^n\to \Rb^m$ is \emph{locally Lipschitz}, if for every $x_0\in \Rb^n$ a constant $L$ and an open set $U$ containing $x_0$ exist, such that
$\|f(x) - f(y)\| \le L\|x -y\|$ for all $x,y\in U$. A \emph{Clarke generalized Jacobian} of $f(\cdot)$ at $x$ is defined as follows
\cite{clarke1975generalized}:
\[
\partial\! f(x) = \conv\big\{ \lim_{\substack{{y\to x}\\{y\in \Dc(f)}}} f'(y)\big\},
\]
where $\Dc(f)$ is the set of points $y$ at which the usual Jacobian $f'(y)$ exists. To simplify further notation,
$\partial\! f(x)$ of a function $f:\Rb^n\to \Rb^m$ is always understood as a set of  $m\times n$ matrices (also for $m=1$).
A locally Lipschitz function
$f:\Rb^n\to\Rb$ is \emph{subdifferentally regular} if for any $x\in \Rb^n$ and $d\in \Rb^n$ the directional derivative
exists and satisfies the equation:
\[
\lim_{\tau\downarrow 0} \frac{f(x+\tau d)-f(x)}{\tau} = \max_{g \in \partial\!f(x)} gd.
\]
\emph{Whitney $\mathcal{C}^1$-stratification} is a partition of a graph of $f(\cdot)$ into finitely many subsets (strata), such that within each of them the function is
continuously differentiable, and a special compatibility condition for the normals on the common parts of the closures of the strata is satisfied
(see, \cite{bolte2007clarke} and the references therein).

We  recall the following definition.

\begin{definition}[\cite{norkin1980generalized}]
\label{d:Norkin}
A function $f:\Rb^n\to\Rb$ is  \emph{differentiable in a generalized sense at a point} $x\in \Rb^n$,
if an open set $U\subset \Rb^n$ containing $x$,
and a nonempty, convex, compact valued, and upper semicontinuous multifunction
$G_f: U \rightrightarrows \Rb^n$ exist, such that for all $y\in U$ and all $g \in G_f(y)$ the following equation is true:
\[
f(y) = f(x) +  g(y-x) + o(x,y,g),
\]
with
\[
\lim_{y\to x} \sup_{g\in G_f(y)} \frac{o(x,y,g)}{\|y-x\|}=0.
\]
The set $G_f(y)$ is  the \emph{generalized derivative} of $f$ at $y$.
If a function is differentiable in a generalized sense at every  $x \in \Rb^n$ with the same generalized derivative $G_f:\Rb^n\rightrightarrows \Rb^n$, we call it
simply \emph{differentiable in a generalized sense}. For functions with values in $\Rb^n$, generalized differentiability is understood component-wise.
\end{definition}

The class of such functions is contained in the set of locally Lipschitz functions,
and contains all semismooth locally Lipschitz functions.
 The generalized derivative $G_f(\cdot)$ is not uniquely defined in Definition \ref{d:Norkin}, which is essential for us,
but the Clarke Jacobian $\partial\!f(x)$ is an inclusion-minimal generalized derivative.
The class of such functions is closed with respect to composition  and
expectation, which allows for easy generation of stochastic subgradients in our case.
%In the Appendix we recall basic properties of Norkin  differentiable functions.
For full exposition, see \cite[Ch. 1 and 6]{mikhalevich1987nonconvex}.

An essential step in the analysis of stochastic recursive algorithms by the differential inclusion method is the
\emph{chain rule on a path} (see \cite{davis2019stochastic}
and the references therein). For an absolutely continuous function $p:[0,\infty)\to\Rb^n$ we denote
by $\dt{p}(\cdot)$ its weak derivative: a measurable function such that
\[
p(t) = p(0) + \int_0^t \dt{p}(s)\;ds,\quad \forall \;  t \ge 0.
\]
\begin{definition}
\label{d:chain}
A locally Lipschitz continuous function $f:\Rb^n \to \Rb^m$ admits a chain rule on absolutely continuous paths, if
\begin{equation}
\label{chain-path}
f(p(T))- f(p(0)) = \int_0^T   g(p(t)) \, \dt{p}(t)  \;dt,
\end{equation}
for any absolutely continuous path $p:[0,\infty)\to \Rb^n$, all selections $g(\cdot) \in \partial\!f(\cdot)$, and all $T>0$.
It admits a chain rule on continuously differentiable paths, if \eqref{chain-path} is true for all continuously differentiable
$p:[0,\infty)\to \Rb^n$, all selections $g(\cdot) \in \partial\!f(\cdot)$, and all $T>0$.
\end{definition}
Formula \eqref{chain-path} is true for convex functions \cite{brezis1971monotonicity} and, as recently demonstrated
in \cite{drusvyatskiy2015curves}, for
subdifferentiably regular locally Lipschitz functions
and Whitney $\mathcal{C}^1$-stratifiable locally Lipschitz functions.
In \cite{Rusz2019} we proved that generalized differentiable functions admit the chain rule on generalized differentiable paths.

To formulate optimality conditions for our problem, and construct and analyze our meth\-od, we need to introduce several relevant multifunctions.
For a point $x\in \Rb^n$ we consider the generalized Jacobians  $\partial\! f_m(\cdot)$, $m=1,\dots,M-1$, and we recursively define the sets and vectors:
\begin{equation}
\label{GF}
\begin{aligned}
G_M(x) &= \partial\! f_M(x),\quad v_M = f_M(x);\\
G_m(x) &= {\rm conv} \big\{ z\in \Rb^{n}: z = g_x + g_u J,\ (g_x,g_u) \in \partial\! f_m(x, v_{m+1}), \ J \in G_{m+1}(x) \big\}, \\
&\quad v_m = f_m(x,v_{m+1}),\quad m= M-1,\dots,1.
\end{aligned}
\end{equation}
By \cite[Thm. 1.6]{mikhalevich1987nonconvex}, %(Theorem \ref{t:composition} in the Appendix),
 each set $G_m(x)$ is a generalized Jacobian of the function
 \begin{equation}
 \label{Fm}
 F_m(x) = f_m\Big(x,f_{m+1}\big(x, \cdots f_{M-1}\big(x,f_M(x)\big)\cdots \big)\Big), \quad m=1,\dots,M,
 \end{equation}
 at $x$.
We also have $\partial F_m(x) \subseteq G_m(x)$ \cite{clarke1981generalized}.
We call a point $x^*\in X$ \emph{stationary} for problem \eqref{main_prob}, if
\begin{equation}
\label{stationary}
0 \in  G_1(x^*) + N_X(x^*).
\end{equation}
The set of stationary points of problem \eqref{main_prob}
is denoted by $X^*$.

\section{The single time-scale method with filtering}
\label{s:3}

The method generates { $M+2$} random sequences: approximate solutions $\{x^k\}_{k\ge 0}\subset \Rb^n$,
path-averaged generalized subgradient estimates $\{z^k\}_{k\ge 0}\subset \Rb^n$, and path-averaged inner functions estimates
$\{u_m^k\}_{k\ge 0}\subset \Rb^{d_{m}}$, $m=1,\dots,M$, all defined
on a certain probability space $(\Omega,\Fc,P)$. We let $\Fc_k$ to be
the $\sigma$-algebra generated by
$\{x^0,\dots,x^k,z^0,\dots,z^k,u^0,\dots,u^k\}$, with each $u^j=(u_1^j,\dots,u_{M}^j)$.
%\begin{assumption}\label{stoch_assump}
%For each $k$, the stochastic oracle delivers random vectors $\tilde{h}^{k+1}\in \Rb^m$, $\tilde{g}^{k+1} \in \Rb^n$,
%and a random matrix $\tilde{J}^{\,k+1} \in \Rb^{m \times n}$, such that
%\begin{align*}
%\tilde{h}^{k+1} &= h(x^{k+1}) + r_1^{k+1},\\
%\tilde{J}^{\,k+1} &\in  \partial h(x^{k+1}) + r_2^{k+1},\\
%\tilde{g}^{k+1} &\in \partial\!f(u^k) + r_3^{k}.
%\end{align*}
%\end{assumption}
%Conditions on the errors $\{r_i^{k}\}$, $i=1,2,3$, will be specified later.

%The shift in indexing of $x^k$ in the above assumption is due to the fact that $x^{k+1}$ will be $\Fc_k$-measurable in our method. The method  proceeds as follows.
%For smooth $f$ we can also have $g(x^k)$ and $g'(x^k)$ on the rhs. }

The method starts from
$x^0 \in X$, $z^0 \in \Rb^n$, $u^0 \in \Rb^{d_1}\times\dots\Rb^{d_M}$, and uses parameters  $a>0$, $b>0$, and $\rho>0$. %positive sequences $\{\rho_k\}_{k \ge 0}$  and $\{\tau_k\}_{k \ge 0} \subset (0,1/a]$.
At each iteration $k=0,1,2,\dots$, we compute
\begin{equation}
\label{QP}
y^k = \argmin_{y \in X}\  \left\{\langle z^k, y-x^k \rangle + \frac{\rho}{2} \|y-x^k\|^2\right\},
\end{equation}
and, with an $\Fc_k$-measurable stepsize $\tau_k \in \big(0,\min(1,1/a,1/b)\big]$, we set
\begin{equation}
\label{def_xk}
x^{k+1} = x^k + \tau_k (y^k-x^k).
\end{equation}
Then, we obtain statistical estimates:
\begin{align*}
& \tilde{J}_m^{k+1}=\begin{bmatrix}\tilde{J}_{mx}^{k+1} & \tilde{J}_{mu}^{k+1}\end{bmatrix}
\text{ of }{J}_m^{k+1}=\begin{bmatrix} {J}_{mx}^{k+1} & {J}_{mu}^{k+1}\end{bmatrix}\in \partial\!f_m(x^{k+1},u_{m+1}^k),\quad m=1,\dots,M-1;\\
& \tilde{J}_M^{k+1}
\text{ of }{J}_M^{k+1}\in \partial\!f_M(x^{k+1});\\
& \tilde{h}_m^{k+1} \text{ of } f_m(x^{k+1},u_{m+1}^k),\quad  m=1,\dots,M-1;\\
& \tilde{h}_M^{k+1} \text{ of } f_M(x^{k+1}).
\end{align*}
We use the stochastic subgradients $\tilde{J}_m^{k+1}$ to construct a biased estimate of a subgradient of the composite function:
\begin{equation}
\label{g-fold}
\tilde{g}_M^{k+1} = \tilde{J}_M^{k+1}, \quad
\tilde{g}_{m}^{k+1} = \tilde{J}_{mx}^{k+1} +  \tilde{J}_{mu}^{k+1} \tilde{g}_{m+1}^{k+1}, \quad m=M-1,\dots,1.
\end{equation}
Finally, we update the path averages by backward recursion as follows:
\begin{align}
z^{k+1} &= z^k + a\tau_k \Big( \big[\tilde{g}_{1}^{k+1}\big]^T  - z^k\Big),  \label{def_zk}\\
u_M^{k+1} &= u_M^k +  \tilde{J}_M^{\,k+1} (x^{k+1}-x^k) + b \tau_k \big(\tilde{h}_M^{k+1}-u_M^k\big),  \label{def_ukM}\\
u_m^{k+1} &= u_m^k +  \tilde{J}_m^{\,k+1} \begin{bmatrix}  x^{k+1}-x^k \\u_{m+1}^{k+1}-u_{m+1}^k \end{bmatrix}
+ b \tau_k \big(\tilde{h}_m^{k+1}-u_m^k\big),\quad  m=M-1,\dots, 1.\label{def_ukm}
\end{align}
%Even in the case of $M=2$ the method differs from the version introduced in \cite{ghadimi2018single} by the
%updates of the sequences $\{u_m^k\}$.

In fact, we do not need the sequence $\{u_1^k\}$ for the operation of the method, but we include it for
uniformity of notation; it will also provide an estimate of the function value at the optimal solution.

We will analyze convergence of the algorithm \eqref{QP}--\eqref{def_ukm} under the following conditions:\vspace{1ex}
\begin{description}
\item[(A1)] The set $X$ is convex and compact;
\item[(A2)] The functions $f_m(\cdot,\cdot)$, $m=1,\dots,M-1$, are Lipschitz continuous and admit the chain rule \eqref{chain-path} for every path $(x(\cdot),u_{m+1}(\cdot))$ with a continuously differentiable $x(\cdot)$ and absolutely continuous $u_{m+1}(\cdot)$.
    Moreover, for every $x\in X$ and $u_{m+1}\in \Rb^{d_{m+1}}$ the $u$-part of the generalized Jacobian, $\partial_u f_m(x,u_{m+1})$, is single-valued.
\item[(A3)] The function $f_M(\cdot)$ is Lipschitz continuous and admits the chain rule \eqref{chain-path} for every continuously differentiable path $x(\cdot)$.
\item[(A4)] A constant $C$ exists, such that   $\|g \| \le C( 1 + \|v\|)$
for all $g \in \partial f_m(x,v)$, all $x\in X$,
and all $v\in \Rb^{d_{m+1}}$, $m=1,\dots,M-1$.
\item[(A5)] The set $\{ F_1(x): x\in X^*\}$ does not contain an interval of nonzero length.
\item[(A6)] $\tau_k \in \big(0,\min(1,1/a,1/b)\big]$ for all $k$,  $\sum_{k=0}^\infty \tau_k = \infty$,
$\sum_{k=0}^\infty \Eb[\tau_k^2] < \infty$.
\item[(A7)] For all $k$,
\begin{tightlist}{iii}
%\item
%$\tilde{g}^{k+1} = g^{k+1} + e_g^{k+1} + \delta_g^{k+1}$, with \\ $g^{k+1}\in \partial\!f(x^{k+1},u^k)$,
% $\Eb\big\{e_g^{k+1}\big|\Fc_k\big\} = 0$, $\Eb\big\{\|e_g^{k+1}\|^2|\Fc_k\big\}\le \sigma_g^2$, $\lim_{k\to \infty} \delta_g^{k+1}=0$,
\item $\tilde{h}_m^{k+1} = f_m(x^{k+1},u_{m+1}^{k}) + e_m^{k+1} + \delta_m^{k+1}$, $m=2,\dots,M-1$, and\\
$\tilde{h}_M^{k+1} = f_M(x^{k+1}) + e_M^{k+1} + \delta_M^{k+1}$, with $\lim_{k\to \infty} \delta_m^{k+1}=0$ and\\
$\Eb\big\{e_m^{k+1}\big|\Fc_k\big\} = 0$,  $\Eb\big\{\|e_m^{k+1}\|^2|\Fc_k\big\}\le \sigma_e^2$,  $m=2,\dots,M$;
\item $\tilde{J}_m^{\,k+1} = J_m^{k+1} + E_m^{k+1} + \Delta_m^{k+1}$,with \\
$J_m^{k+1}\in \partial f_m(x^{k+1},u_{m+1}^k)$, $m=1,\dots,M-1$, $J_M^{k+1}\in \partial f_M(x^{k+1})$,\\
$\Eb\big\{E_m^{k+1}\big|\Fc_k\big\} = 0$, $\Eb\big\{\|E_m^{k+1}\|^2|\Fc_k\big\}\le \sigma_E^2$, $\lim_{k\to \infty}  \Delta_m^{k+1}=0$;
\end{tightlist}
\item[(A8)] For $m=2,\dots,M-1$, the error $ E_{mu}^{k+1} $ is conditionally independent of $E_\ell^{k+1}$ and $e_\ell^{k+1}$, $\ell=m+1,\dots,M$, given $\Fc_k$.
\end{description}

\begin{remark}
\label{r:chain}
Assumption (A2) is satisfied if
\[
f_m(x,u_{m+1}) = \varphi_m(\psi_m(x),u_{m+1}),\quad m= 1,\dots, M-1,
\]
where $\varphi_m(\cdot,\cdot)$ is continuously differentiable, and $\psi_m(\cdot)$ is differentiable in a generalized sense.
Indeed, by virtue of \cite[Thm. 1]{Rusz2019}, if a path $x(t)$ is continuously differentiable, the function $\psi_m(\cdot)$ admits
the chain rule on this path.
This implies that the path $\psi_m(x(t))$ is absolutely continuous. Consequently, $\varphi(\psi_m(x(t)),u_{m+1}(t))$
admits the chain rule. { This is true in both Examples \ref{e:1} ($p>1$) and \ref{e:SVI}, with a generalized differentiable
function (operator) $\Eb[H(\cdot)]$. }
%In particular, the operator $H(\cdot)$ in Example \ref{e:1} may represent the training error of a ReLU network.}
\end{remark}

\section{Basic Properties}
\label{s:4}

As the calculation of generalized subgradients of a composition requires the knowledge of the values of the inner functions,
we will also need a more general  multifunction
$\varGamma:\Rb^n\times\Rb^n\times\Rb^{d_{1}}\times \dots \times \Rb^{d_{M}} \rightrightarrows \Rb^n\times\Rb^{d_{1}}\times \dots \times \Rb^{d_{M}}$:
% {\color{red} avoid convex hull}
%\begin{equation}
%\label{GFu}
%\begin{aligned}
%\varGamma_M(x,u) &= D_M(x);\\
%\varGamma_m(x,u) &= {\rm conv} \big\{ z\in \Rb^{n}: z = D_m\begin{bmatrix} I\\ J_{m+1}\end{bmatrix},\ D_m \in \partial\! f_m(x, u_{m+1}),
% \ J_{m+1} \in \varGamma_{m+1}(x,u) \big\},\\
%& \quad m= M-1,\dots,1.
%\end{aligned}
%\end{equation}
%{  better version}
%\begin{equation}
%\label{Gamma-new}
%\begin{aligned}
%\varGamma_M(x,z,u_M) &= \partial\!f_M(x)  \begin{bmatrix} \;\bar{y}(x,z)-x  &  b\big( f_M(x)- u_M\big)\;\end{bmatrix};\\
%\varGamma_m(x,z,u_m,\dots,u_M) &= \bigg\{  \partial\! f_m(x, u_{m+1})\begin{bmatrix} I & \bar{y}(x,z)-x\\ J_{m+1} & v_{m+1}\end{bmatrix}
%+  \begin{bmatrix}\; 0 & b\big(  f_m(x,u_{m+1})- u_m\big)\; \end{bmatrix}: \\
%&\qquad \begin{bmatrix}J_{m+1} & v_{m+1} \end{bmatrix} \in \varGamma_{m+1}(x,u) \bigg\},
% \quad m= 1,\dots,M-1.
%\end{aligned}
%\end{equation}
\begin{equation}
\label{Gamma-new-2}
\begin{aligned}
\varGamma(x,z,u) &= \Big\{ (g,v_1,\dots,v_M): \exists D_m\in \partial f_m(x,u_{m+1}), \;\exists D_M\in \partial f_M(x),\\
& { v_M = D_{M}\big(\bar{y}(x,z)-x\big) + b\big(  f_M(x)- u_M\big)},\\
& v_m = D_{mx}\big(\bar{y}(x,z)-x\big) + D_{mu} v_{m+1}+ b\big(  f_m(x,u_{m+1})- u_m\big), \; m=1,\dots,M-1,\\
& g_M = D_M,\;g_m = D_{mx} + D_{mu} g_{m+1},\;m=1,\dots,M-1,\; g = a (g_1^T-z) \Big\}.
\end{aligned}
\end{equation}
{ Here,
\[
\bar{y}(x,z) = \argmin_{y\in X} \left\{\langle z, y-x \rangle + \frac{\rho}{2} \|y-x\|^2\right\}.
\]}
The multifunction $\varGamma(\cdot)$ is convex and compact valued, due to assumptions (A2) and (A3).

%Define the sequence of multifunctions
%\[
%H_m:\Rb^n\times\Rb^n\times \Rb^{d_{m+1}}\times \dots \times \Rb^{d_M} \rightrightarrows \Rb^{d_m}, \quad m= 1,\dots,M,
%\]
%as follows:
%\begin{align}
%H_M(x,z,u_M) &=  \Big\{  D_M (\bar{y}(x,z)-x )  + b\big( f_M(x)- u_M\big): D_M \in \partial\!f_M(x)\Big\}, \label{HM-def-new}\\
%\intertext{and for $m=M-1,\dots,1$ {\color{red} avoid convex hull}}
%H_m(x,z,u_{m},\dots,u_M) &= \conv \bigg\{  D_m \begin{bmatrix} \bar{y}(x,z)-x \\ v_{m+1}\end{bmatrix}  + b\big(  f_m(x,u_{m+1})- u_m\big):
%\label{H-def-new}\\
%& \qquad D_m \in \partial\!f_m(x,u_{m+1}),\  v_{m+1} \in H_{m+1}(x,z,u_{m+1},\dots,u_M)\bigg\} \notag
%\end{align}
%{ We may assume that $\partial f_m(x,u_{m+1})$ is-single valued in the $u$-coordinate.}
%\begin{multline}
%\label{H-def}
%H_m(x,z,u_{m+1},\dots,u_M) = \conv \bigg\{ v\in \Rb^d: v_M \in \partial\!f_M(x) \big[\bar{y}(x,z)-x\big] + b f_M(x),\\
%v_m \in  \partial\!f_m(x,u_{m+1}) \begin{bmatrix} \bar{y}(x,z)-x \\ v_{m+1}- b u_{m+1}\end{bmatrix}  + b f_m(x,u_{m+1}), \ m = M-1,\dots,1\bigg\}.
%\end{multline}
This multifunction allows us to write the iterations of $\{u^k_m\}$ and $\{z^k\}$ in a more compact way, and to establish
the boundedness of these sequences. In our analysis we consider the $x$- and $u$-components of $\varGamma(\cdot)$ separately,
by writing $g \in \varGamma_x(x,z,u)$ and $v_m\in \varGamma_{mu}(x,z,u)$,
but it is essential to keep in mind that they all derive from the same multifunction (same $D_m$'s).
\begin{lemma}
\label{l:uk-bounded}
If conditions {\rm (A1)} and {\rm (A6)--(A8)} are satisfied, then for all $k$ the following relations are true:
\begin{equation}
\label{u-abstract}
u_m^{k+1} \in  u_m^k + \tau_k \varGamma_{mu}(x^{k+1},z^k,u^k) + \tau_k \theta_m^{k+1} + \tau_k \alpha_m^{k+1},
\end{equation}
%{ better}
%\begin{equation}
%\label{uz-abstract}
%\begin{bmatrix} z_m^{k+1} & u_m^{k+1} \end{bmatrix} \in  \begin{bmatrix} z_m^k & u_m^k\end{bmatrix}
%+ \tau_k \varGamma_m(x^{k+1},z^k,u_m^k,\dots,u_M^k) + \tau_k \theta_m^{k+1} + \tau_k \alpha_m^{k+1},
%\end{equation}
where, for some constant $C^{\theta}_m$,
\begin{equation}
\label{theta-abstract}
\Eb\big[\theta_m^{k+1}\,\big|\, \Fc_k\big]=0,  \quad \Eb\big[\|\theta_m^{k+1}\|^2\,\big|\,\Fc_k\big]\le C^{\theta}_m, \quad  k=0,1,\dots
\end{equation}
and
\begin{equation}
\label{alpha-abstract}
\lim_{k\to \infty} \alpha_m^{k+1} = 0 \quad \text{a.s.}.
\end{equation}
Moreover, the sequence $\{u^k\}$ is bounded a.s..
\end{lemma}
\begin{proof}
For $m=M$, formula \eqref{def_ukM} and assumption (A7) yield:
\begin{equation}
\label{ukM}
u_M^{k+1}  =  u_M^k + \tau_k v_M^k + \tau_k \theta_M^{k+1} + \tau_k \alpha_M^{k+1},
\end{equation}
with
\begin{gather*}
v_M^k =  J_M^{k+1} \big(\bar{y}(x^k,z^k)-x^k\big) + b \big( f_M(x^{k+1})- u_M^k\big),\\
\theta_M^{k+1} =  E_M^{k+1}\big(\bar{y}(x^k,z^k)-x^k\big) + b e_M^{k+1},\\
\alpha_M^{k+1} = \Delta_M^{k+1} \big(\bar{y}(x^k,z^k)-x^k\big) + b \delta_M^{k+1}.
\end{gather*}
Due to (A1) and (A7), relations \eqref{u-abstract}--\eqref{alpha-abstract} are true for coordinate $m=M$.

To verify the boundedness of $\{u_M^k\}$, we define the quantities
\begin{equation}
\label{tildeu}
\tilde{u}_M^k = u_M^k +  \sum_{j=k}^\infty \tau_j \theta_M^{j+1}.
\end{equation}
Owing to (A6) and  \eqref{theta-abstract},  by virtue of the martingale convergence theorem,
the series in the formula above is convergent a.s., and thus $\tilde{u}_M^k - u_M^k \to 0$ a.s., when $k\to \infty$.
We can now use \eqref{ukM} to establish the following recursive relation:
\begin{equation}
\label{tildeu-simple}
\tilde{u}_M^{k+1} = (1-b\tau_k) \tilde{u}_M^k + \tau_k J_M^{k+1} \big(\bar{y}(x^k,z^k)-x^k\big)
+ b \tau_k f_M(x^{k+1})+ \tau_k \alpha_M^{k+1} \\
 + b\tau_k (\tilde{u}_M^k -u_M^k).
\end{equation}
By (A1), the sequences $\{J_M^k\}$ and $\{f_M(x^k)\}$ are bounded, and thus
\[
\limsup_{k\to \infty} \big\| {u}_M^k \big\| = \limsup_{k\to \infty} \big\| \tilde{u}_M^k \big\|
\le \limsup_{k\to \infty} \big\| J_M^{k+1} \big(\bar{y}(x^k,z^k)-x^k\big)
+ b  f_M(x^{k+1}) \big\| < \infty.
\]
Furthermore, it follows from \eqref{tildeu} and \eqref{tildeu-simple} for $M$ that
 a constant $C_u$ exists, such that \break $\Eb\big[ \|u_m^k\|^2\| \big] \le C_u$.

We now proceed by induction. Suppose the relations \eqref{u-abstract}-\eqref{alpha-abstract} are true for $m+1$,
the sequence $\{u^k_{m+1}\}$  is bounded a.s., and
$\Eb\big[ \|u_{m+1}^k\|^2 \big]$  is bounded as well.  We shall verify these properties
for $m$. From \eqref{def_ukm} for $m$ and \eqref{u-abstract} for $m+1$ we obtain
\[
u_m^{k+1} = u_m^k + \tau_k \tilde{J}_m^{\,k+1} \begin{bmatrix}  \bar{y}(x^k,z^k)-x^k \\v_{m+1}^{k} + \theta_{m+1}^{k+1} + \alpha_{m+1}^{k+1}\end{bmatrix}
+ b \tau_k \big(\tilde{h}_m^{k+1}-u_m^k\big).
\]
This can be rewritten as follows:
\[
u_m^{k+1} =  u_m^k +  \tau_k v^k_m + \tau_k \theta_M^{k+1} + \tau_k \alpha_M^{k+1},
\]
with
\begin{gather}
v_m^k =  J_m^{k+1} \begin{bmatrix} \bar{y}(x^k,z^k)-x^k \\ v^k_{m+1} \end{bmatrix} + b\big(  f_m(x^{k+1},u^k_{m+1}) - u^k_m\big), \label{uk-ind}\\
\theta_m^{k+1} = E_m^{k+1}\begin{bmatrix} \bar{y}(x^k,z^k)-x^k \\ v^k_{m+1} \end{bmatrix}
+ \big(J_{mu}^{k+1} + E_{mu}^{k+1}\big) \theta_{m+1}^{k+1}
 + E_{mu}^{k+1} \alpha_{m+1}^{k+1} + b e_m^{k+1},
\label{theta-ind}\\
\alpha_m^{k+1} = J_{mu}^{k+1} \alpha_{m+1}^{k+1}
+ \Delta_m^{k+1}\begin{bmatrix} \bar{y}(x^k,z^k)-x^k \\ v^k_{m+1} +\alpha_{m+1}^{k+1} \end{bmatrix}  + b \delta_m^{k+1}. \label{alpha-ind}
\end{gather}
As the sequences $\{x^k\}$ and $\{u_{m+1}^k\}$ are bounded a.s., the Jacobians $J_m^{k+1}$  are bounded as well. Moreover, by assumption (A5), $\Eb\big[\| J_m^{k+1}\|^2\big]$ is bounded for all $k$.

To verify \eqref{theta-abstract}--\eqref{alpha-abstract},
we only need to analyze the effect of various product terms in the formulae above. The product
$J_{mu}^{k+1} v^k_{m+1}$ in \eqref{uk-ind} is square integrable due to (A4). The same is true
for $f_m(x^{k+1},u^k_{m+1})$. The conditional expectation $\Eb\big[\theta_m^{k+1}\,\big|\,\Fc_k\big] = 0 $, thanks to
assumption (A8), because $\theta^{k+1}_{m+1}$ depends only on observations of quantities associated with the functions $f_\ell(\cdot,\cdot)$,
$\ell=m+1,\dots,M$. Furthermore, $\Eb\big[\|\theta_m^{k+1}\|^2\,\big|\,\Fc_k\big]$ is bounded, due to (A4), (A7), and (A8).
Consequently, we can define a sequence $\{\tilde{u}_m^k\}$ is a way analogous to \eqref{tildeu},
and establish for it a recursive relation of the form \eqref{tildeu-simple}, with $M$ replaced by $m$. In the same way as above, we
obtain the boundedness of $\{{u}_m^k\}$ with probability 1, and its square integrability.
By induction, the assertion is true for all $m$.
\end{proof}

We now pass to the analysis of the sequence $\{z^k\}$. Carrying out \eqref{g-fold} for exact subgradients, we would obtain
\begin{equation}
\label{g-fold-exact}
{g}_M^{k+1} = {J}_M^{k+1}, \quad {g}_{m}^{k+1} = J_{mx}^{k+1}+J_{mu}^{k+1}{g}_{m+1}^{k+1}, \quad m=M-1,\dots,1.
\end{equation}
Evidently, $a \big(\big[g_1^{k+1}\big]^T-z^k\big) \in \varGamma_{x}\big(x^{k+1},z^k,u^k\big)$.

\begin{lemma}
\label{l:zk-recursion}
If conditions {\rm (A1)} and {\rm (A6)--(A8)} are satisfied, then for all $k=0,1,\dots$ and all $m=1,\dots,M$,
\begin{equation}
\label{gradient-error}
\tilde{g}_m^{k+1} = g_m^{k+1} + \varkappa_m^{k+1} + \beta_m^{k+1}, \quad m=1,\dots,M,
\end{equation}
where, for some constant $C_\varkappa$,
\begin{equation}
\label{gradient-error-2}
\Eb\big[\varkappa_m^{k+1} \,\big|\, \Fc_k\big] = 0, \quad \Eb\big[\|\varkappa_m^{k+1}\|^2 \,\big|\, \Fc_k\big] \le C_\varkappa, \quad \lim_{k\to \infty} \beta_m^{k+1} = 0.
\end{equation}
\end{lemma}
\begin{proof}
Formulae \eqref{gradient-error}--\eqref{gradient-error-2} are true for $m=M$ directly by (A7)(ii). Supposing that they are
true for $m+1$, we verify them for $m$. We have
\begin{multline}
\label{gradient-error-3}
\tilde{g}_m^{k+1} - g_m^{k+1} =
 \tilde{J}_{mx}^{k+1} - J_{mx}^{k+1}
 +\tilde{J}_{mu}^{k+1}\tilde{g}_{m+1}^{k+1}
-J_{mu}^{k+1} {g}_{m+1}^{k+1}\\
 =
 \tilde{J}_{mx}^{k+1} - J_{mx}^{k+1}
 +\big(\tilde{J}_{mu}^{k+1} -J_{mu}^{k+1} \big)\tilde{g}_{m+1}^{k+1}
+ J_{mu}^{k+1} \big( \tilde{g}_{m+1}^{k+1} -{g}_{m+1}^{k+1} \big).
\end{multline}
The first term in \eqref{gradient-error-3}, $\tilde{J}_{mx}^{k+1} - {J}_{mx}^{k+1}$, admits decomposition of form
\eqref{gradient-error}--\eqref{gradient-error-2} directly by (A7)(ii). The second term, by (A7)(ii),  can be represented as follows:
\[
\big(\tilde{J}_{mu}^{k+1} -{J}_{mu}^{k+1}\big) \tilde{g}_{m+1}^{k+1} =
\big({E}_{mu}^{k+1} + \Delta_{mu}^{k+1}\big) \big({g}_{m+1}^{k+1} + \varkappa_{m+1}^{k+1} + \beta_{m+1}^{k+1}\big).
\]
Owing to (A8),
\[
\Eb \big\{  {E}_{mu}^{k+1} \varkappa_{m+1}^{k+1}\,\big|\, \Fc_k\big\} =0,\quad
\Eb \big\{  \big\|{E}_{mu}^{k+1}  \varkappa_{m+1}^{k+1}\big\|^2 \,\big|\, \Fc_k\big\} \le C_\varkappa\sigma_E^2.
\]
Together with the boundedness of $\{x^k\}$ and $\{u^k\}$, this implies that the second term admits  decomposition od form \eqref{gradient-error}--\eqref{gradient-error-2}.
The third term in \eqref{gradient-error-3}, by  \eqref{gradient-error},  can be represented as follows:
\[
 {J}_{mu}^{k+1} \big( \tilde{g}_{m+1}^{k+1} -{g}_{m+1}^{k+1} \big)
 =   {J}_{mu}^{k+1} \big( \varkappa_{m+1}^{k+1} + \beta_{m+1}^{k+1}\big).
\]
Together with the boundedness of $\{J_m^k\}$, this implies that the third term admits  decomposition od form \eqref{gradient-error}--\eqref{gradient-error-2}
as well. Therefore, \eqref{gradient-error}--\eqref{gradient-error-2} are true for all $m$.
\end{proof}

We can now establish the boundedness of the sequence $\{z^k\}$.

\begin{lemma}
\label{l:zk-bounded}
If conditions {\rm (A1)} and {\rm (A6)--(A8)} are satisfied, then
with probability~1 the sequence  $\{z^k\}$ is bounded.%, $z^{k+1}-z^k \to 0$ and $u_m^{k+1}-u_m^k \to 0$, $m=1,\dots,M$}.
\end{lemma}
\begin{proof}
We proceed as in the proof of lemma \ref{l:uk-bounded}. We define the quantities
\[
\tilde{z}^k  = z^k +  a\sum_{j=k}^\infty \tau_j \big[\varkappa_1^{j+1}\big]^T, \quad k=0,1,\dots.
\]
Due to Lemma \ref{l:zk-recursion},  by virtue of the martingale convergence theorem,
the series in the formula above is convergent
a.s., and thus $\tilde{z}^k - z^k \to 0$ a.s., when $k\to \infty$.
We can now use \eqref{def_zk} to establish the following recursive relation:
\[
\tilde{z}^{k+1} = (1-a\tau_k) \tilde{z}^k
+  a\tau_k \big[g_1^{k+1}\big]^T +  a\tau_k(\big[\beta_1^{k+1}\big]^T+\tilde{z}^k -z^k).
\]
Therefore,
\[
\limsup_{k\to \infty} \big\| {z}^k \big\| = \limsup_{k\to \infty} \big\| \tilde{z}^k \big\| \le \limsup_{k\to \infty} \big\| g_1^k \big\| < \infty\quad \text{a.s.},
\]
as claimed.
\end{proof}

\section{Convergence analysis}
\label{s:analysis}

%{ Clarify optimality conditions}

We start from a useful property of the gap function {$\eta:X\times \Rb^n\to(-\infty,0]$},
\begin{equation}
\label{gap}
\eta(x,z) = \min_{y\in X} \left\{\langle z, y-x \rangle + \frac{\rho}{2} \|y-x\|^2\right\}.
\end{equation}
We denote the minimizer in \eqref{gap} by $\bar{y}(x,z)$. Since it is a projection of $x-z/\rho$ on $X$,
\begin{equation}
\label{opt-eta}
 \langle z ,\bar{y}(x,z)-x \rangle  + \rho \| \bar{y}(x,z)-x \|^2 \le 0.
\end{equation}
Moreover, a point $x^*\in X^*$ if and only if $z^*\in G_1(x^*)$ exists
 such that $\eta(x^*,z^*)=0$.

We can now state the main result of the paper.

\begin{theorem}
\label{t:convergence}
If the assumptions {\rm (A1)--(A8)} are satisfied,
then with probability 1 every accumulation point $\hat{x}$ of the sequence $\{x^k\}$ is stationary, $\lim_{k\to\infty} (u_m^k-F_m(x^k))=0$,
for $m=1,\dots,M$, and the sequence $\{F_1(x^k)\}$ is convergent.
\end{theorem}
\begin{proof}
%By virtue of Lemma \ref{l:zk-bounded}, the sequences $\{z^k\}$ and $\{u^k\} are bounded.
 We consider a specific trajectory of the method and divide the proof into three standard steps.

\emph{Step 1: The Limiting Dynamical System.}  We denote by $p^k=(x^k,z^k,u^k)$, $k=0,1,2,\dots$, a realization
of the sequence generated by the algorithm.
We introduce the accumulated stepsizes
$t_k = \sum_{j=0}^{k-1}\tau_j$, $k=0,1,2 \dots$, and we construct the interpolated trajectory
\[
P_0(t) = p^k + \frac{t-t_{k}}{\tau_k}(p^{k+1}-p^k),\quad t_{k}\le t \le t_{k+1},\quad k=0,1,2,\dots.
\]
For an increasing sequence of positive numbers $\{s_k\}$ diverging to $\infty$, we define shifted trajectories $P_k(t) =P_0(t+s_k)$.
 The sequence $\{p^k\}$ is bounded by Lemmas \ref{l:uk-bounded} and  \ref{l:zk-bounded}.

By \cite[Thm. 3.2]{majewski2018analysis},  for any infinite set $\Kc$ of positive integers,
there exist an infinite  subset $\Kc_1 \subset \Kc$ and an absolutely continuous function  $P_\infty:
[0,+\infty) \to X\times \Rb^n\times \Rb^m$ such that for any $T > 0$
\[
\lim_{\substack{{k\to\infty}\\{k\in \Kc_1}}}\sup_{t\in[0,T]} \big\| P_k(t)-P_\infty(t)\big\|= 0,
\]
and $P_\infty(\cdot)=\big(X_\infty(\cdot),Z_\infty(\cdot),U_\infty(\cdot)\big)$
is a solution of the system of differential equations and inclusions corresponding to \eqref{def_xk}, \eqref{def_zk} with
\eqref{gradient-error}, and \eqref{u-abstract}:
\begin{gather}
\dt{x}(t) = \bar{y}\big(x(t),z(t)\big)-x(t), \label{dx}\\
\big(\dt{z}(t), \dt{u}(t)\big) \in \varGamma(x(t),z(t),u(t)). \label{dz}
%\dt{u}(t) &\in \varGamma_{u}(x(t),z(t),u(t)). \label{du}
\end{gather}
Moreover, for any $t\ge 0$, the triple $\big(X_\infty(t),Z_\infty(t),U_\infty(t)\big)$
is an accumulation point of the sequence $\{(x^k,z^k,u^k)\}$.

%\begin{remark}
%\label{r:joint}
%We write \eqref{dz} and \eqref{du} as separate inclusions, but we should
%keep in mind that they use the same multifunction $\varGamma(\cdot)$, whose selections will share the
%Jacobians $J_m(x(t),u_{m+1}(t))$. We could have replaced $z$ by vectors $z_m$, $m=1,\dots,M$ (representing the estimates
%of the subgradients of $F_m(\cdot)$), and written joint inclusions involving $\begin{bmatrix} z_m & u_m\end{bmatrix}$
%and mutlifunctions $\varGamma_m(\cdot)$, but this would additionally convolute our presentation.
%\end{remark}

In order to analyze the equilibrium points of the system \eqref{dx}--\eqref{dz}, we first study the dynamics
of the functions $\varPhi_m(t) =  f_m(X(t),U_{m+1}(t))$, $m =1,\dots,M-1$, and $\varPhi_M(t) = f_M(X(t))$. { In what follows,
the equations and inequalities involving $X(t)$, $U_m(t)$. and $\varPhi_m(t)$ are understood as holding for almost all $t \ge 0$.}

It follows from \eqref{dx} that the path $X(\cdot)$ is continuously differentiable. By virtue of assumption (A3),
for any $J_M(t) \in \partial f_M(X(t))$,
\begin{equation}
\label{f-incr-M}
 \dt \varPhi_M(t)
=   J_{M}(t) \dt X(t) .
\end{equation}
Assumption (A2) means that for any $J_m(t) \in \partial f_m(X(t),U_{m+1}(t))$,
\begin{equation}
\label{f-incr-m-2}
 \dt \varPhi_m(t)
=   J_{m}(t) \begin{bmatrix} \dt X(t)  \\  \dt U_{m+1}(t) \end{bmatrix}, \quad m=1,\dots,M-1.
\end{equation}
We need to understand the dynamics of $U_m(\cdot)$. From \eqref{dz} and \eqref{Gamma-new-2} we deduce that
\begin{equation}
\label{UM}
\dt U_M(t) = \hat{J}_M(t) \dt X(t) + b [\varPhi_M(t) - U_M(t)],
\end{equation}
with some $\hat{J}_M(t) \in \partial f_M(X(t))$, and
\[
\dt U_m(t) = \hat{J}_m(t) \begin{bmatrix} \dt X(t) \\ \dt U_{m+1}(t)\end{bmatrix}  + b[ \varPhi_m(t) - U_m(t)], \quad m=1,\dots,M-1,
\]
with some $\hat{J}_m(t) \in \partial f_m(X(t),U_{m+1}(t))$.
Therefore, using $J_m(\cdot) = \hat{J}_m(\cdot)$, for $m=1,\dots,M$, in \eqref{f-incr-m-2}, we obtain
\begin{equation}
\label{Um}
\dt U_m(t) =  \dt \varPhi_m(t) + b[ \varPhi_m(t) - U_m(t)].
\end{equation}
We can verify by induction that the solution of \eqref{f-incr-m-2}--\eqref{Um} has the form:
 \begin{equation}
 \label{Phi-sol}
   \dt \varPhi_m(t) = \hat{g}_m(t)\dt X(t)
   + b \sum_{\ell=m}^{M-1} \prod_{q=m}^\ell \hat{J}_{q u}(t)\big[\varPhi_{\ell +1}(t)-U_{\ell +1}(t)\big],
   \quad m=1,\dots,M,
 \end{equation}
with $\hat{g}_m(t)$ defined by the recursive procedure:
\[
\hat{g}_M(t) =  { \hat{J}_{M}}(t), \quad \hat{g}_m(t) =  \hat{J}_{mx}(t) + \hat{J}_{mu}(t) \hat{g}_{m+1}(t),\quad m=M-1,\dots,1,
\]
and $\hat{J}_{mx}$ and $\hat{J}_{mu}$ denoting the $x$-part and the $u$-part of $\hat{J}_m$, respectively.
These observations will help us study the stability of the system.

\emph{Step 2: Descent Along a Path.}  We use the Lyapunov function
\begin{equation}
\label{Lyapunov}
W(x,z,u) = a f_1(x,u_2) -  \eta(x,z) + \sum_{m=2}^{M-1} \gamma_m\, \big\|f_m(x,u_{m+1})-u_m\big\| + \gamma_M \big\|f_M(x)-u_M\big\|,
\end{equation}
with the coefficients $\gamma_m >0$ to be specified later.
Directly from \eqref{Phi-sol} for $m=1$ we obtain
\begin{multline}
\label{f-incr2}
  f_1(X(T),U_2(T)) - f_1(X(0),U_2(0)) \\
 = \int_0^T  \hat{g}_1(t) \dt{X}(t)  \;dt
  + b  \sum_{\ell=1}^{M-1} \int_0^T \prod_{q=1}^\ell \hat{J}_{q u}(t)\big[\varPhi_{\ell +1}(t)-U_{\ell +1}(t)\big]\;dt.
\end{multline}

We now estimate the change of $\eta(X(\cdot),Z(\cdot))$ from 0 to $T$.
 Since $\bar{y}(x,z)$ is unique, the function $\eta(\cdot,\cdot)$ is continuously differentiable.
Therefore, the chain formula holds:
\begin{multline*}
\eta(X(T),Z(T)) - \eta(X(0),Z(0)) \\
= \int_0^T  \big\langle \nabla_x \eta(X(t),Z(t)), \dt{X}(t) \big\rangle \;dt +
\int_0^T  \big\langle \nabla_z \eta(X(t),Z(t)), \dt{Z}(t) \big\rangle \;dt.
\end{multline*}
From \eqref{dz} and \eqref{g-fold-exact}, we obtain { (for almost all $t\ge 0$)}
\[
\dt{Z}(t) = a\big(\hat{g}_1^T(t) - Z(t)\big),
\]
with the same $\hat{g}_1(\cdot)$ as in \eqref{Phi-sol} for $m=1$ and in \eqref{f-incr2}.
% { This is due to the fact that
%\eqref{dz}--\eqref{du} are in fact two components of one differential inclusion involving the multifunctions $\varGamma_m(\cdot)$
%(see Remark \ref{r:joint}), and every selection of $\varGamma_m(\cdot)$ uses the same Jacobians $\hat{J}_m$ in both its
%components.}

Substituting
$\nabla_x \eta(x,z) = -z+\rho(x- \bar{y}(x,z))$,
$\nabla_z \eta(x,z) =  \bar{y}(x,z)-x$,  and
using \eqref{opt-eta}, we obtain
\begin{align*}
\lefteqn{\eta(X(T),Z(T)) - \eta(X(0),Z(0))}\quad  \\
&= \int_0^T  \big\langle -Z(t)+\rho(X(t)-\bar y(X(t),Z(t)))\, , \,\bar{y}(X(t),Z(t)) - X(t) \big\rangle \;dt \\
&{\quad } +
a \int_0^T  \big\langle \bar y(X(t),Z(t))-X(t)\,,\,\hat{g}^T_1(t) - Z(t) \big\rangle \;dt\\
&\ge\; a \int_0^T  \big\langle \bar y(X(t),Z(t))-X(t)\,,\, \hat{g}_1^T(t) - Z(t) \big\rangle \;dt\\
&\ge\; a \int_0^T  \hat{g}_1(t)\big( \bar y(X(t),Z(t))-X(t)\big) \;dt
{}  + a\rho \int_0^T  \big\| \bar y(X(t),Z(t))-X(t)\big\|^2 \;dt.
   \end{align*}
With a view at \eqref{dx}, we conclude that
\begin{equation}
\label{eta-incr}
 \eta(X(T),Z(T)) - \eta(X(0),Z(0))
  \ge \;  a \int_0^T  \hat{g}_1(t) \dt X(t) \;dt
 + a\rho \int_0^T  \big\| \dt X(t)\big\|^2 \;dt.
  \end{equation}
% We obtain
%\begin{multline*}
%a \big[ f(U(T)) - f(U(0))\big] - [\eta(X(T),Z(T),U(T)) - \eta(X(0),Z(0),U(0))]\\
%\le b a \int_0^T \big\langle g_f(U(t)), H(t)- U(t) \big\rangle \;dt - a\rho \int_0^T  \big\| \dt X(t)\big\|^2 \;dt \\
%\le b a L \int_0^T \|H(t)-U(t)\|\;dt - a\rho \int_0^T  \big\| \dt X(t)\big\|^2 \;dt ,
%\end{multline*}
%where $L$ is the Lipschitz constant of $f(\cdot)$.

We now estimate the increment of $\big\| \varPhi_m(\cdot)-U_m(\cdot)\big\|$ from 0 to $T$. As
 $\| \cdot\| $ is convex and $\varPhi_m(\cdot)$ and $U_m(\cdot)$ are paths, the chain rule applies as well: for any
 $\lambda_m(t) \in \partial \| \varPhi_m(t)- U_m(t)\|$ we have
\[
 \big\|\varPhi_m(T)-U_m(T)\big\| -  \big\|\varPhi_m(0)-U_m(0)\big\|
= \int_0^T \big\langle \lambda_m(t),  \dt \varPhi_m(t) - \dt U_m(t) \big\rangle\; dt.
\]
By \eqref{Um},  $\dt \varPhi_m(t) - \dt U_m(t) = b\big[ U_m(t)-\varPhi_m(t)\big]$ for almost all $t \ge 0$.
Furthermore,
\[
\lambda_m(t)=\frac{\varPhi_m(t)-U_m(t)}{\|\varPhi_m(t)-U_m(t)\|},\quad \text{if}\quad  \varPhi_m(t)\ne U_m(t).
\]
Therefore
\begin{equation}
\label{norm-incr}
 \big\|\varPhi_m(T)-U_m(T)\big\| -  \big\|\varPhi_m(0)-U_m(0)\big\| = - b \int_0^T \big\| \varPhi_m(t)-U_m(t)\big\| \; dt.
 \end{equation}

We can now combine \eqref{f-incr2}, \eqref{eta-incr}, and \eqref{norm-incr} to estimate the change of the Lyapunov function \eqref{Lyapunov}:
\begin{multline*}
%\label{W-incr}
W\big(X(T),Z(T),U(T)\big) - W\big(X(0),Z(0),U(0)\big)\quad \\
\quad \le   a  b  \sum_{m=1}^{M-1} \int_0^T \prod_{q=1}^m \hat{J}_{q u}(t)\big[\varPhi_{m +1}(t)-U_{m +1}(t)\big]\;dt\\
 - a\rho \int_0^T  \big\| \dt X(t)\big\|^2 \;dt
- b\sum_{m=2}^{M} \gamma_m  \int_0^T \big\|\varPhi_{m}(t)-U_{m}(t)\big\|\;dt.
\end{multline*}
Because the paths $X(t)$ and $U(\cdot)$ are bounded a.s. and the functions $f_m$ are locally Lipschitz,  a (random) constant $L$ exists, such that
$\big \| J_{m u}(t)\big\| \le L$ for $m=1,\dots,M-1$.
The last estimate entails:
\begin{multline}
\label{W-incr}
W\big(X(T),Z(T),U(T)\big) - W\big(X(0),Z(0),U(0)\big) \\
 \le - a\rho \int_0^T  \big\| \dt X(t)\big\|^2 \;dt - b \sum_{m=2}^{M} (\gamma_m-aL^{m-1})  \int_0^T \|\varPhi_m(t)-U_m(t)\|\;dt.
\end{multline}
By choosing $\gamma_m > aL^{m-1}$ for all $m$, we ensure that $W(\cdot)$ has the descent property to be used in our stability analysis at Step 3.
The fact that $L$ (and thus $\gamma_m$) may be different for different paths is irrelevant, because our analysis is path-wise.

\emph{Step 3: Analysis of the Limit Points.} Define the set
\begin{multline*}
\quad \Sc = \big\{ (x,z,u)\in X^* \times \Rb^n\times \Rb^m: \eta(x,z)=0,\\
 u_m=f_m(x,u_{m+1}), \ m=1,\dots,M-1,\ u_M=f_M(x)\big\}.\quad
\end{multline*}
Suppose $(\bar{x},\bar{z},\bar{u})$ is an accumulation point of the sequence $\{(x^k,z^k,u^k)\}$. If $\eta(\bar{x},\bar{z}) <0$ or
$\bar{u}_m\ne f_m(\bar{x},\bar{u}_{m+1})$, then
every solution $(X(t),Z(t),U(t))$ of the system \eqref{dx}--\eqref{dz},
starting from $(\bar{x},\bar{z},\bar{u})$ has $\|\dt X(0)\| > 0$ or $\|\varPhi_m(0)-U_m(0)\| >0$.
Using \eqref{W-incr} and arguing as in \cite[Thm. 3.20]{duchi2018stochastic} or \cite[Thm. 3.5]{majewski2018analysis},
we obtain a contradiction. Therefore,
we must have $\eta(\bar{x},\bar{z})=0$ and $\bar{u}_m= f_m(\bar{x},\bar{u}_{m+1})$, $m=1,\dots,M-1$, and
$\bar{u}_M= f_M(\bar{x})$. Suppose $\bar{x}\not \in X^*$. Then
%, with
%\[
%D = \left( I , \big[\partial h(\bar{x})\big]^T\right) \partial\!f(\bar{x},h(\bar{x})) ),
%\]
%we have
\begin{equation}
\label{non-opt}
\dist\big(0,  G_1(\bar{x}) + N_X(\bar{x})\big) >0.
\end{equation}
Suppose the system \eqref{dx}--\eqref{dz}
starts from $(\bar{x},\bar{z},\bar{u})$ and $X(t)=\bar{x}$ for all $t\ge 0$. From \eqref{dz} and \eqref{Gamma-new-2}, in view of the
equations $\bar{y}(\bar{x},\bar{z})=\bar{x}$ and  $\bar{u}_M=f_M(\bar{x})$, we obtain  $U_M(t)=  f_M(\bar{x})$ for all $t \ge 0$. Proceeding by backward induction,
we conclude that $U_m(t)=  f_m(\bar{x},\bar{u}_{m+1})$ for all $t \ge 0$ and all $m=1,\dots,M-1$.
The inclusion \eqref{dz}, in view of \eqref{GF}, simplifies
\[
\dt{z}(t) \in a \big( G_1(\bar{x}) - z(t)\big).
\]
For the convex Lyapunov function $V(z) = \dist\big(z,G_1(\bar{x})\big)$,
we apply the classical chain formula \cite{brezis1971monotonicity}
on the path $Z(\cdot)$:
\[
V((Z(T)) - V(Z(0)) = \int_0^T \big\langle \partial V(Z(t)), \dt Z(t)\big\rangle \;dt.
\]
For $Z(t)\notin G_1(\bar{x})$, we have
\[
\partial V(Z(t)) = \frac{Z(t) - \proj_{G_1(\bar{x})}(Z(t))}{ \| Z(t) - \proj_{G_1(\bar{x})}(Z(t))\|}
\]
 and $\dt Z(t) = a (d(t) - Z(t))$ with some $d(t)\in G_1(\bar{x})$. Therefore,
\[
\big\langle \partial V(Z(t)), \dt Z(t)\big\rangle \le  - a \| Z(t) - \proj_{G_1(\bar{x})}(Z(t))\| = -a V(Z(t)).
\]
It follows that
\[
V((Z(T)) - V(Z(0)) \le - a \int_0^T  V(Z(t)) \;dt,
\]
and thus
\begin{equation}
\label{Z-conv}
\lim_{t\to\infty} \dist\big(Z(t),G_1(\bar{x})\big) = 0.
\end{equation}
It follows from \eqref{non-opt}--\eqref{Z-conv} that $T>0$ exists,
such that $ -Z(T) \not \in N_X(\bar{x})$, which yields $\dt X(T)\ne 0$. Consequently,
the path $X(t)$ starting from $\bar{x}$ cannot be constant (our supposition made right after \eqref{non-opt} cannot be true).  But if it is not constant, then again $T>0$ exists, such that $\dt X(T)\ne 0$. By Step 1,
the triple $(X(T),Z(T),U(T))$ would have to be an accumulation point of the sequence $\{(x^k,z^k,u^k)\}$, a case already excluded.
We conclude that every accumulation point $(\bar{x},\bar{z},\bar{u})$  of the sequence $\{(x^k,z^k,u^k)\}$ is in $\Sc$.
The convergence of the sequence $\big\{W(x^k,z^k,u^k)\big\}$ then follows in the same way as
\cite[Thm. 3.20]{duchi2018stochastic} or \cite[Thm. 3.5]{majewski2018analysis}. As $\eta(x^k,z^k)\to 0$,
the convergence of $\{f_1(x^k,u_2^k)\}$ follows as well. Since $f_m(x^k)-u_m^k \to 0$, $m=2,\dots,M$, the sequence $\{F_1(x^k)\}$ is convergent
as well.
\end{proof}
Adapting the proof of Lemma \ref{l:zk-bounded}, we obtain convergence of path-averaged stochastic subgradients.
\begin{corollary}
\label{c:conv-zk}
If the sequence $\{x^k\}$ is convergent to a single point $\bar{x}$, then every accumulation point
of $\{z^k\}$ is an element of the generalized gradient $G_{F_1}(\bar{x})$ satisfying the optimality condition at $\bar{x}$.
\end{corollary}

In fact, if we introduced path-averaging of the vectors $\tilde{g}_{m}^{k}$ used in \eqref{g-fold}, in a way similar to
\eqref{def_zk}:
\[
z_m^{k+1} = z_m^k + a\tau_k \Big( \big[\tilde{g}_{m}^{k+1}\big]^T  - z_m^k\Big),  \quad m=1,\dots,M,
\]
we could extend Corollary \ref{c:conv-zk} to the convergence of each $\{z_m^k\}$ to an element
of the generalized gradient $G_{F_m}(\bar{x})$.

\section{Fixed Stepsize Performance}
\label{s:5}

Although our main interest is in nonsmooth problems, we present here performance guarantees after finitely many steps with a constant stepsize, when the functions are continuously differentiable.
We make stronger assumptions about the functions  $f_m$ in problem \eqref{main_prob} and about the estimation bias in (A6):
\begin{description}
\item[(A9)] The functions $f_m(\cdot,\cdot)$, $m=1,\dots,M-1$, and $f_M(\cdot)$ are continuously differentiable with Lipschitz continuous derivatives;
\item[(A10)] A constant $C_\delta>0$ exists, such that $\max\big(\|\delta_m^{k+1}\|,\|\Delta_m^{k+1}\|\big) \le C_\delta \tau_k$ with probability 1, for all $k$ and $m=1,\dots,M$;
\item[(A11)] A constant $L$ exists, such that
$\| {J}_{mu}^{k+1}\| \le L$ with probability 1, for all $k=0,1,2\dots$.
\end{description}

\vspace{0.5ex}

In (A10) and (A11) we need $C_\delta$ and $L$ to be the same for all paths, because we need one Lyapunov function for all paths.
Assumptions (A9)--(A11) imply (A2)--(A4).
We also do not need assumptions (A5) and (A6), since we are interested in finitely many iterations.

Our reasoning is similar to the line of argument in \cite[Sec. 3]{ghadimi2018single},
albeit with the multi-level nested structure. We assume that we carry out a finite number $N$ of iterations with a constant stepsize $\tau>0$.
A similar analysis of the rate of convergence of a closely related method has been recently provided in
\cite[Sec. 3]{balasubramanian2020stochastic}, albeit under strong assumptions about fourth moments of the errors.

In our considerations below, $C$ is a sufficiently large deterministic constant dependent on various constants featuring in the
problem definition, such as the Lipschitz constants, the variances of the errors, etc., but not dependent on the number of iterations $N$,
algorithm parameters, or the stepsize $\tau$.
%We define the sequence
%\begin{equation}
%\label{Lyapunov-discrete}
%W_k = a f(x^k,u^k) -  \eta(x^k,z^k) + \gamma\, \|h(x^{k+1})-u^k\|, \quad k=0,1,2,\dots.
%\end{equation}
%The shift of the index in the last term does not introduce complications, because $x^{k+1}$ is a function of $x^k$, $z^k$, and $\tau_k$, and thus the
%sequence $\{W_k\}$ is adapted to $\{\Fc_k\}$.

{ We shall use a different Lyapunov function than \eqref{Lyapunov}, namely
\begin{equation}
\label{Lyapunov-2}
\Wc(x,z,u) = a F_1(x) -  \eta(x,z) + \sum_{m=2}^{M-1} \gamma_m\, \big\|f_m(x,u_{m+1})-u_m\big\|^2 + \gamma_M \big\|f_M(x)-u_M\big\|^2.
\end{equation}
It is a direct generalization to the multi-level case of the function used in \cite{ghadimi2018single}.
The reason is that a smooth function is more amenable for fixed step analysis.}
We shall need the following condition about the
function \eqref{Lyapunov-2}.
\begin{description}
\item[(A12)] Constants $\gamma_m^{\min}$ exist, for $m=2,\dots,M$, such that if all $\gamma_m \ge \gamma_m^{\min}$ then the function \eqref{Lyapunov-2} is bounded from below on $X\times\Rb^n\times\Rb^{d_2}\times \dots \times \Rb^{d_M}$.
\end{description}

{
By using assumptions (A7)--(A10), and adapting to discrete time  our  analysis from Step~2 of the proof of Theorem \ref{t:convergence},
we estimate the three terms of the difference
\begin{multline*}
\Eb\big[\Wc(x^{k+1},z^{k+1},u^{k+1})\,\big|\, \Fc_k\big] - \Wc(x^k,z^k,u^k)\\
=
a \big( \Eb\big[F_1(x^{k+1})\,\big|\, \Fc_k\big] - F_1(x^k)\big)
 - \big( \Eb\big[\eta(x^{k+1},z^{k+1})\,\big|\, \Fc_k\big] - \eta(x^k,z^k) \big)\\
 +
\sum_{m=2}^{M-1} \gamma_m \big( \Eb\big[ \|f_m(x^{k+1},u_{m+1}^{k+1})-u_m^{k+1} \|^2 \,\big|\, \Fc_k\big] - \|f_m(x^{k},u_{m+1}^{k})-u_m^{k} \|^2\big) \\
+ \gamma_M \big( \Eb\big[ \|f_M(x^{k+1})-u_M^{k+1} \|^2 \,\big|\, \Fc_k\big] - \|f_M(x^{k})-u_M^{k} \|^2\big).
\end{multline*}

%For simplicity, we set $a=b=1$, albeit similar calculations can
%be carried out for general $a$ and $b$.

Denote $d^k=y^k-x^k$.
The decrease of the first part of the Lyapunov function can be estimated as follows (with an adjustment of the constant $C$):
\begin{equation}
\label{f-incr-dis}
\begin{aligned}
 \Eb\big[F_1(x^{k+1})\,\big|\, \Fc_k\big] - F_1(x^k)
&\le {\tau} \hat{g}_1^{k+1} d^k + \tau \|\hat{g}_1^{k+1} - \hat{g}_1^{k}\| \|d^k\|
  +  \tau \|G_1(x^k) - \hat{g}_1^{k}\| \|d^k\|
  + C {\tau}^2 \\
  &\le {\tau} \hat{g}_1^{k+1} d^k
  +  \tau \|G_1(x^k) - \hat{g}_1^{k}\| \|d^k\|
  + C^{\text{new}} {\tau}^2 .
  \end{aligned}
\end{equation}
Recall that for continuously differentiable functions
\[
\hat{g}_M^k =  f'_{M}(x^k), \quad \hat{g}_m^k =   f'_{mx}(x^k,u_{m+1}^k) +  f'_{mu}(x^k,u_{m+1}^k) \hat{g}_{m+1}^k,\quad m=M-1,\dots,1,
\]
and
\[
G_m(x^k) = f'_{M}(x^k),\quad G_m^k =   f'_{mx}\big(x^k,F_{m+1}(x^k)\big) +  f'_{mu}\big(x^k,F_{m+1}(x^k)\big) G_{m+1}^k,\quad m=M-1,\dots,1.
\]
Let $L_G$ be an upper bound on the Lipschitz constants of all Jacobians featuring in the formulas above.
Proceeding by induction backward in $m$, we can verify that
\[
\big\|G_1(x^k) - \hat{g}_1^{k}\big\| \le \sum_{m=1}^{M-1} L_G^{m}\,\big\| F_{m+1}(x^k) - u_{m+1}^k\big\|.
\]
Furthermore, for $m=2,\dots,M-1$, in a similar way we obtain
\begin{align*}
\big\| F_{m}(x^k) - u_{m}^k\big\| &\le \big\| f_m(x^k,u_{m+1}^k) - u_m^k\big\|  + L \big\| F_{m+1}(x^k) - u_{m+1}^k\big\|\\
&\le \sum_{q=m}^M L^{q-m} \big\| f_q(x^k,u_{q+1}^k) - u_q^k\big\|.
\end{align*}
Therefore,
\[
\big\|G_1(x^k) - \hat{g}_1^{k}\big\| \le \sum_{m=2}^{M} L_G^{m}\sum_{q=m}^M L^{q-m} \big\| f_q(x^k,u_{q+1}^k) - u_q^k\big\|
= \sum_{q=2}^{M} \sum_{m=2}^q L_G^{m} L^{q-m} \big\| f_q(x^k,u_{q+1}^k) - u_q^k\big\|.
\]
Substitution into \eqref{f-incr-dis} yields (with $K_q = \sum_{m=2}^q L_G^{m} L^{q-m}$)
\begin{equation}
\label{f-incr-dis-2}
 \Eb\big[F_1(x^{k+1})\,\big|\, \Fc_k\big] - F_1(x^k)
\le {\tau} \hat{g}_1^{k+1} d^k
  +  \tau \|d^k\| \sum_{q=2}^{M} K_q \,\big\| f_q(x^k,u_{q+1}^k) - u_q^k\big\|
  + C {\tau}^2 .
\end{equation}
}

Due to (A6) and the continuous differentiability of the function $\eta(\cdot,\cdot)$, the change of the second part of the Lyapunov function can be
estimated as follows (the inequalities are derived from \eqref{opt-eta}, as in \eqref{eta-incr}):
\begin{equation}
\label{eta-incr-dis}
 \Eb\big[ \eta(x^{k+1},z^{k+1}) - \eta(x^{k},z^{k})\big]
  \ge \; a {\tau}    \hat{g}_1^{k+1} d^k   + a\rho {\tau}  \big\| d^k \big\|^2 - C{\tau}^2.
  \end{equation}
  We remark here that the Lipschitz constant of the gradient of $\eta$ has a uniform bound for all $\rho>0$, and thus
  we may assume that $C$ is independent of $\rho$.
 Finally, we analyze the decrease of the third part, associated with the ``tracking errors'' $f_m(x^k,u_{m+1}^k) - u_m^k$, for $m= 1,\dots,M$ and $k=0,1,...,N$. { When $m=M$, the argument $u_{m+1}$ is not present, but we keep uniform notation for brevity.}

\begin{lemma}
\label{l:tracking}
A constant $C$ exists, such that
if the method carries out $N$ iterations with a constant stepsize $\tau \in \big(0,1/(4b)\big)$, then
for  $m=1,2\dots,M$,
\begin{multline}
\label{fm-drop}
 \Eb\big[ \|f_m(x^{N},u_{m+1}^{N})-u_m^{N}\|^2\big]-
  \|f_m(x^{0},u_{m+1}^0)-u_m^{0}\|^2  \\
  \le
-\frac{1}{2} b \tau \sum_{k=0}^{N-1}  \Eb\big[\|f_m(x^{k},u_{m+1}^k)-u_m^{k}\|^2\big]   + C N \tau^2.
\end{multline}
%\begin{equation}
%\label{power-two}
%\sum_{k=0}^N  \Eb\big[\|f_m(x^{k})-u_m^{k}\|^2\big]\le
% \frac{2}{\tau b}\Big( \|f_m(x^{0})-u_m^{0}\|^2 + C N \tau^2\Big),
%\end{equation}
\end{lemma}
\begin{proof}
First, we consider the term associated with $f_M(\cdot)$ to expose our ideas.
Expanding the square, we obtain
\begin{multline}
\label{H-first}
\big\|f_M(x^{k+1})-u_M^{k+1}\big\|^2 \\
= \big\|f_M(x^{k})-u_M^{k}\big\|^2 + 2 \big\langle f_M(x^{k})-u_M^{k} ,  f_M(x^{k+1})-u_M^{k+1} - f_M(x^{k}) + u_M^{k} \big\rangle \\
 + \big\|f_M(x^{k+1})-u_M^{k+1} - f_M(x^{k}) + u_M^{k}\big\|^2.
\end{multline}
 Due to assumption (A9),
\[
f_M(x^{k+1})- f_M(x^k) = {\tau} J_M^{k+1} d^k + r_M^{k+1},
\]
with $\Eb\big[\|r_M^{k+1}\| \,\big|\, \Fc_k\big] \le C {\tau}^2$. By \eqref{def_ukM},
\begin{align*}
u_M^{k+1} - u_M^k &= {\tau} \tilde{J}^{\,k+1}d^k + b {\tau} \big(\tilde{h}_M^{k+1}-u_M^k\big) \\
&= {\tau} {J}^{k+1}d^k + {\tau} E_M^{k+1}d^k + {\tau} \Delta_M^{k+1} d^k + b {\tau} (f_M(x^{k+1}) - u^k) + b {\tau} e_M^{k+1} + b{\tau} \delta_M^{k+1}.
\end{align*}
Therefore
\begin{align*}
\lefteqn{f_M(x^{k+1})- f_M(x^k) - u_M^{k+1} + u_M^k}\quad\\
&= -b{\tau} (f_M(x^{k+1}) - u_M^k)- {\tau} E_M^{k+1}d^k - {\tau} \Delta^{k+1} d^k - b {\tau} e_h^{k+1} - b{\tau} \delta_h^{k+1}+  r_M^{k+1} \\
&= -b{\tau} (f_M(x^{k}) - u_M^{k})  - {\tau} E_M^{k+1}d^k  - b {\tau} e_h^{k+1}+  r^{k+1},
\end{align*}
where $\|r^{k+1}\| \le C {\tau}^2$.
Then
\[
\Eb\big[\big\| f_M(x^{k+1})- f_M(x^k) - u_M^{k+1} + u_M^k\big\|^2\,\big|\,\Fc_k\big] \le 2b {\tau}^2 \big\| f_M(x^{k}) - u_M^{k}\big\|^2 + 2 C {\tau}^2.
\]
We can now substitute the last  two expressions into the scalar product in \eqref{H-first} and take the conditional expectation of both sides with respect to $\Fc_k$:
\[
 \Eb\big[ \|f_M(x^{k+1})-u_M^{k+1}\|^2 \,\big|\, \Fc_k\big]  \le (1- b {\tau} +2 b^2{\tau}^2)\|f_M(x^{k})-u_M^{k}\|^2   + C {\tau}^2.
\]
Adding these inequalities for $k=0,1,\dots,N-1$ and taking the expected value of both sides, we get:
\[
 \Eb\big[ \|f_M(x^{N})-u_M^{N}\|^2\big]-
  \|f_M(x^{0})-u_M^{0}\|^2 \le
-b \tau(1-2 b\tau) \sum_{k=0}^N  \Eb\big[\|f_M(x^{k})-u_M^{k}\|^2\big]   + C N \tau^2.
\]
With $\tau < 1/(4b)$, we obtain \eqref{fm-drop} for $m=M$.

%Next, by Jensen's inequality and \eqref{power-two},
%\begin{align*}
%%\label{power-down}
%\sum_{k=0}^N  \Eb \big[\|f_M(x^{k})-u_M^{k}\|\big] &\le  \sum_{k=0}^N   \Big( \Eb\big[\|f_M(x^{k})-u_M^{k}\|^2\big]\Big)^{\frac{1}{2}}\\
%&\le\sqrt{N} \bigg( \sum_{k=0}^N   \Eb\big[\|f_M(x^{k})-u_M^{k}\|^2\big]\bigg)^{\frac{1}{2}}
%\le \sqrt{\frac{2N}{\tau b}}\Big( \|f_M(x^{0})-u_M^{0}\|^2 + C N \tau^2\Big)^{\frac{1}{2}}.
%\end{align*}
%This is \eqref{power-down} for $m=M$.

For $1 \le m < M$ the analysis follows the same steps. The expansion \eqref{H-first} is similar:
\begin{multline*}
\big\|f_m(x^{k+1},u_{m+1}^{k+1})-u_M^{k+1}\big\|^2 =  \big\|f_m(x^{k},u_{m+1}^k)-u_m^{k}\big\|^2 \\
 +
2 \big\langle f_m(x^{k},u_{m+1}^k)-u_m^{k} ,  f_m(x^{k+1},u_{m+1}^{k+1})-u_m^{k+1} - f_m(x^{k},u_{m+1}^k) + u_m^{k} \big\rangle \\
 + \big\|f_m(x^{k+1},u_{m+1}^{k+1})-u_m^{k+1} - f_m(x^{k},u_{m+1}^k) + u_m^{k}\big\|^2.
\end{multline*}
The only difference is the presence of the arguments $u_{m+1}^{k}$ and $u_{m+1}^{k+1}$. Due to assumption (A9),
\[
 f_m(x^{k+1},u_{m+1}^{k+1})- f_m(x^{k},u_{m+1}^k) = {\tau} J_{mx}^{k+1} d^k + J_{mu}^{k+1}(u_{m+1}^{k+1}-u_{m+1}^{k}) + r_m^{k+1}.
\]
The additional term, $J_{mu}^{k+1}(u_{m+1}^{k+1}-u_{m+1}^{k})$ is compensated (with stochastic errors) by the corresponding term in the update
of $u_m$ by \eqref{def_ukm}:
\[
u_m^{k+1} - u_m^k = {\tau} \tilde{J}_{mx}^{\,k+1} d^k + \tilde{J}_{mu}^{\,k+1} (u_{m+1}^{k+1}-u_{m+1}^{k})  + b {\tau} \big(\tilde{h}_m^{k+1}-u_m^k\big).
\]
The remaining steps are identical and lead to the estimate  \eqref{fm-drop} for $m<M$.
\end{proof}

Integrating \eqref{f-incr-dis-2}, \eqref{eta-incr-dis}, and \eqref{fm-drop}, we obtain the following inequality (we write $\varPhi_m^k$ for $f_m(x^k,u_{m+1}^k)$ and $\varPhi_M^k$ for $f_M(x^k)$):
{
\begin{multline*}
\Eb\big[ \Wc(x^{k+1},z^{k+1},u^{k+1})\,\big|\,\Fc_k\big]  - \Wc(x^{k},z^{k},u^{k}) \le -a\rho {\tau}\| d^k\|^2  \\
+ a  \tau \|d^k\| \sum_{m=2}^{M} K_{m}\,\big\| \varPhi_{m}^{k} - u_{m}^k\big\|
- \frac{1}{2} b {\tau} \sum_{m=2}^{M} \gamma_m \|\varPhi_m^{k}-u_m^{k}\|^2
 + C {\tau}^2 \\
 \le - \tau \sum_{m=2}^{M} \bigg( \ \frac{a\rho}{M-1} \| d^k\|^2
 - aK_m \| d^k\| \|\varPhi_m^{k}-u_m^{k}\|
 +  \frac{b\gamma_m}{2}\|\varPhi_m^{k}-u_m^{k}\|^2 \bigg)  + C {\tau}^2 ,
\end{multline*}
where $C$  is a constant independent of $\tau$ and $\gamma_m$. If
\[
\gamma_m > \frac{a K_m^2(M-1)}{2b\rho}, \quad m=2,\dots,M,
\]
then the quadratic forms in the parentheses are positive definite. Then $\delta>0$ exists, such that
\[
\label{W-discrete}
\Eb\big[ \Wc(x^{k+1},z^{k+1},u^{k+1})\,\big|\,\Fc_k\big]  - \Wc(x^{k},z^{k},u^{k}) \le -\delta {\tau}\| d^k\|^2
- \delta {\tau} \sum_{m=2}^{M} \|\varPhi_m^{k}-u_m^{k}\|^2
 + C {\tau}^2.
\]
Taking the expected value of both sides and summing for $k$ from 0 to $N-1$, we get
\begin{align*}
\lefteqn{\Eb\big[ \Wc(x^{N},z^{N},u^{N})\big]  - \Wc(x^{0},z^{0},u^{0})}\quad \\
&\le \Eb\bigg[ -\delta {\tau} \sum_{k=0}^{N-1}\| y^k-x^k\|^2
- \delta {\tau} \sum_{m=2}^{M} \sum_{k=0}^{N-1} \|\varPhi_m^{k}-u_m^{k}\|^2 \bigg]
 + CN {\tau}^2.
\end{align*}
}
Let $\Wc_{\min}$ be the minimum value of $\Wc(\cdot,\cdot,\cdot)$, existing by virtue of (A12). Dividing the last inequality by $\tau\delta N$
and simplifying, we obtain
\[
 \frac{1}{N} \sum_{k=0}^{N-1}\Eb \bigg[\| y^k-x^k\|^2 +   \sum_{m=2}^{M} \|f_m(x^{k},u_{m+1}^k)-u_m^{k}\|^2 \bigg]\\
 \le \frac{1}{\tau\delta N} \big( \Wc(x^{0},z^{0},u^{0}) - \Wc_{\min}\big) + \frac{C\tau}{\delta}.
\]
With $\tau \sim N^{-1/2}$,
\[
 \frac{1}{N} \sum_{k=0}^{N-1} \Eb\bigg[  \| y^k-x^k\|^2 +  \sum_{m=2}^{M}  \|f_m(x^{k},u_{m+1}^k)-u_m^{k}\|^2  \bigg] \\
\le \frac{\text{Const}}{\sqrt{N}}.
\]
Therefore, at a random iteration $R$, drawn from the uniform distribution in $\{0,1,\dots, N-1\}$, independently from other random quantities in the method,
we have
\begin{equation}
\label{rate}
\Eb\bigg[ \| y^R-x^R\|^2 + \sum_{m=2}^{M}\|f_m(x^{R},u_{m+1}^R)-u_m^{R}\| \bigg] \le \frac{\text{Const}}{\sqrt{N}}.
\end{equation}
No conditions on the method's parameters are required for this result, as opposed to our earlier analysis of the two-level case in \cite{ghadimi2018single}. This improvement is due to the linear correction terms in \eqref{def_ukM}--\eqref{def_ukm}, accounting for the movement of $x$.

It also follows from \eqref{fm-drop} that we can write an estimate similar to \eqref{rate} for each individual tracking error.
For each $m=1,\dots,M$,
\begin{equation}
\label{track-rate}
\Eb\big[\|f_m(x^{R},u_{m+1}^R)-u_m^{R}\|^2\big] \le  \frac{2}{b \sqrt{N}}\|f_m(x^{0},u_{m+1}^0)-u_m^{0}\|^2 + \frac{C}{\sqrt{N}}.
\end{equation}

If we consider the quantities on the left hand side of \eqref{rate} and \eqref{track-rate}
as measures of non-optimality and tracking errors of the triple $(x^R,z^R,u^R)$, we can deduce  that $\mathcal{O}(1/\varepsilon^2)$ iterations are needed to achieve an error of
size $\varepsilon$. The reference \cite[Sec. 2]{ghadimi2018single} provides a thorough discussion of similar optimality measures used for constrained
stochastic optimization.

\section*{Acknowledgments} The author thanks Saeed Ghadimi for finding a mistake in an earlier version of \S \ref{s:5}.

\end{document}